\documentclass[reqno,12pt]{amsart}
\usepackage[colorlinks=true, linkcolor=blue, citecolor=blue]{hyperref}

\usepackage{amssymb}
\usepackage{amsmath, graphicx, rotating}
\usepackage{color}
\usepackage{soul}
\usepackage[dvipsnames]{xcolor}

\usepackage{ifthen}
\usepackage{xkeyval}
\usepackage{todonotes}
\usepackage[T1]{fontenc}
\usepackage{lmodern}
\usepackage[english]{babel}

\usepackage{ upgreek }
\usepackage{stmaryrd}
\SetSymbolFont{stmry}{bold}{U}{stmry}{m}{n}
\usepackage{amsthm}
\usepackage{float}

\usepackage{ bbm }
\usepackage{ stmaryrd }
\usepackage{ mathrsfs }
\usepackage{ frcursive }
\usepackage{ comment }

\usepackage{pgf, tikz}
\usetikzlibrary{shapes}
\usepackage{varioref}
\usepackage{enumitem}

\definecolor{rouge}{rgb}{0.7,0.00,0.00}
\definecolor{vert}{rgb}{0.00,0.5,0.00}
\definecolor{bleu}{rgb}{0.00,0.00,0.8}
\usepackage[margin=1in]{geometry}
\newtheorem{theorem}{Theorem}[section]
\newtheorem*{theorem*}{Theorem}
\newtheorem{lemma}[theorem]{Lemma}

\newtheorem{proposition}[theorem]{Proposition}

\labelformat{hypothesis}{\textbf{M\kern-0.1mm#1}}

\newtheorem{condition}{Condition}

\newtheorem{conditionA}{A\kern-0.1mm}
\labelformat{conditionA}{\textbf{A\kern-0.1mm#1}}

\theoremstyle{definition}

\newtheorem{remark}[theorem]{Remark}

\def \eref#1{\hbox{(\ref{#1})}}

\numberwithin{equation}{section}

\def\geq{\geqslant}
\def\leq{\leqslant}

\def\RR{\mathbb{R}}

\def\EE{\mathbb{E}}

\def\vare{{\varepsilon}}
\def \eref#1{\hbox{(\ref{#1})}}

\def\EE{\mathbb{ E}}

\DeclareMathOperator{\e}{e}

\begin{document}

\title[Averaging principle for slow-fast SPDEs with H\"{o}lder continuous coefficients]
{ Averaging principle for slow-fast stochastic partial differential equations with H\"{o}lder continuous coefficients}

\author{Xiaobin Sun}
\curraddr[Sun, X.]{ School of Mathematics and Statistics, Jiangsu Normal University, Xuzhou, 221116, China}
\email{xbsun@jsnu.edu.cn}

\author{Longjie Xie}
\curraddr[Xie, L.]{ School of Mathematics and Statistics, Jiangsu Normal University, Xuzhou, 221116, China}
\email{longjiexie@jsnu.edu.cn}

\author{Yingchao Xie}
\curraddr[Xie, Y.]{ School of Mathematics and Statistics, Jiangsu Normal University, Xuzhou, 221116, China}
\email{ycxie@jsnu.edu.cn}

\begin{abstract}
By using the technique of the Zvonkin's transformation and the classical Khasminkii's time discretization method, we prove the averaging principle for slow-fast stochastic partial differential equations with bounded and H\"{o}lder continuous drift coefficients. An example is also provided to explain our result.
\end{abstract}
\date{}

\maketitle
\section{Introduction}

In this paper, we consider the following stochastic partial differential equation in a Hilbert space $H$:
\begin{equation}\left\{\begin{array}{l}\label{Main equation}
\displaystyle
dX^{\varepsilon}_t=\left[A X^{\varepsilon}_t
+B(X^{\varepsilon}_t, Y^{\varepsilon}_t)\right]dt+\sqrt{Q_1}d W^1_{t},\\
\displaystyle dY^{\varepsilon}_t=\frac{1}{\varepsilon}\left[AY^{\varepsilon}_t+F(X^{\varepsilon}_t, Y^{\varepsilon}_t)\right]dt
+\frac{1}{\sqrt{\varepsilon}}\sqrt{Q_2}d W^2_{t},\\
\end{array}\right.
\end{equation}
where $\varepsilon>0$ is a small parameter describing the ratio of the time scales of the slow component $X^{\varepsilon}_t$
and the fast component $Y^{\varepsilon}_t$, $A: \mathscr{D}(A)\to H$ is the infinitesimal generator of a linear
strongly continuous semigroup $\{e^{tA}\}_{t\geq 0}$, $B$ and $F$ are appropriate continuous functions, $Q_1$ and $Q_2$ are two non-negative selfadjoint bounded operators in $H$, $\{W^{1}_t\}_{t\geq 0}$ and $\{W^{2}_t\}_{t\geq 0}$ are $H$-valued mutually independent  cylindrical Wiener processes defined on a complete filtered probability space
$(\Omega,\mathscr{F},\{\mathscr{F}_{t}\}_{t\geq0},\mathbb{P})$.

\vspace{0.1cm}
The multiscale system \eref{Main equation}  has wide applications in material sciences,
fluid dynamics, biology, ecology, climate dynamics, see e.g., \cite{BR,EE,HKW,Ku} and the references therein.
The averaging principle is essential to describe the asymptotic behavior of the system as $\vare\to 0$, which says that the slow component will converge to the so-called averaged equation. This theory was first developed for the  ordinary differential equations (ODEs for short) by Bogoliubov and Mitropolsky \cite{BM}, and extended to the stochastic differential equations (SDEs for short) by Khasminskii \cite{K1}, see also \cite{L,LRSX1}. Since the averaging principle for a general class of stochastic reaction-diffusion systems with two time-scales was investigated by Cerrai and Freidlin in \cite{CF}, the averaging principle for slow-fast stochastic partial differential equations (SPDEs for short) has been drawn much attentions in the past decades, see  e.g., \cite{B1,C1,C2,CL,DSXZ,FLL,FWLL,GP,GP1,GP2,LRSX1,LRSX2,SZ,WR12,WRD12} and the references therein.

\vspace{0.1cm}
We point out that all the references  mentioned above assumed that the coupled coefficients $B$ and $F$ satisfy at least local Lipshitz continuous condition. However, it was shown by Da Prato and Flandoli \cite{DF} that system (\ref{Main equation}) can be strongly well-posed with only H\"older continuous drift coefficients, see also \cite{D-F-P-R,D-F-P-R-2} for further generalizations. Thus it is natural to ask that whether the averaging principle still holds under such kind of conditions. As far as we know, the averaging principle for stochastic system with irregular coefficients has not been studied  much yet. Even in the case of SDEs, there are only few results in this direction. Veretennikov \cite{V0} studied the averaging principle for SDEs under the assumptions that the drift coefficient of slow equation is only bounded and measurable with respect to slow variable, and all the other coefficients are global Lipschitz continuous. R\"ockner and authors \cite{RSX} studied the strong and weak convergence in the averaging principle for SDEs with H\"{o}lder coefficients drift, also see \cite{PV1,PV2} for the study of diffusion approximations for SDEs with singular coefficients.

\vspace{0.1cm}
The main purpose of this paper is to prove the strong convergence in the averaging principle for SPDE (\ref{Main equation}) with bounded and H\"{o}lder coefficients, i.e., for any $p\geq 1$,
\begin{align*}
\lim_{\vare\rightarrow 0}\mathbb{E} \left(\sup_{t\in[0,T]}|X_{t}^{\vare}-\bar{X}_{t}|^{p} \right)=0,
\end{align*}
where $\bar{X}_t$ is the solution of the corresponding averaged equation (see \eref{1.3} below). To the best of our knowledge, this seems to be the first paper which studies the averaging principle of slow-fast SPDEs with irregular coefficients. Nevertheless the SDEs with H\"{o}lder coefficients have been studied in \cite{RSX}, it will be quite different from the case of SPDE. For example, the main method used in \cite{RSX} is based on solving the Poisson equation with singular coefficient in finite dimension. However it is difficulty and nontrivial to study the Poisson equation with singular coefficient in infinite dimensional case. As a result, inspired from \cite{V0}, we intend to use the classical Khasminskii's time discretization and combine the Zvonkin's transformation, the latter is now widely used to study the strong well-posedness for S(P)DEs with singular coefficients, see e.g. \cite{DF,D-F-P-R,D-F-P-R-2,Kr-Ro}.

\vspace{1mm}
The rest of the paper is organized as follows. In Section \ref{Sec Main Result}, we first give some notations and suitable assumptions, then we present our main result and give a direct-viewing the idea of the key technique. Section \ref{Sec proof of result} is devoted to proving our main result. In Section \ref{Sec Exs}, we will give an example to illustrate the applicability of our result.

\vspace{1mm}
Throughout the paper, $C$, $C_p$, $C_T$ and $C_{p,T}$ denote positive constants which may change from line to line, where the subscript $p, T$ are used to emphasize  that the constant only depends on the parameters $p, T$.

\section{Notations and main results} \label{Sec Main Result}

\subsection{Notations and assumptions}

Let us first introduce some notations. The inner product and the norm of $H$, which are denoted by $\langle\cdot,\cdot\rangle$ and $|\cdot|$ respectively. We assume the following conditions throughout the paper:
\begin{conditionA}\label{A1}
$B, F: H\times H\rightarrow H$ are measurable and bounded. Moreover, there exist constants $\alpha,\beta,\gamma\in (0,1]$ and $C>0$ such that for any $x_1,x_2,y_1,y_2\in H$,
\begin{eqnarray*}
&&|B(x_1, y_1)-B(x_2, y_2)|\leq C\left(|x_1-x_2|^{\alpha} + |y_1-y_2|^{\beta}\right),\\
&&|F(x_1, y_1)-F(x_2, y_2)|\leq C|x_1-x_2|^{\gamma}+ L_F|y_1-y_2|.
\end{eqnarray*}
\end{conditionA}

\begin{conditionA}\label{A2}
$A$ is a selfadjoint operator satisfying $Ae_k=-\lambda_k e_k$ with $\lambda_k>0$ and $\lambda_k\uparrow \infty$, as $k\uparrow \infty$, where $\{e_k\}_{k\geq1}\subset \mathscr{D}(A)$ is a complete orthonormal basis of $H$.
\end{conditionA}

\begin{conditionA}\label{A3}
There exists $\zeta\in (0,1)$ such that $\sum_{k\geq1}\lambda_k^{\zeta-1}<\infty.$
\end{conditionA}

\begin{conditionA}\label{A4}
There exists $\theta\in (0,1)$ such that for any $T>0$,
\begin{eqnarray}
&&\int^T_0 r^{-\theta}\|e^{rA}\sqrt{Q_1}\|^2_{HS}dr\leq C_T,\label{A41}\\
&&\int^T_0 \|(-A)^{\theta/2}e^{rA}\sqrt{Q_1}\|^2_{HS}dr\leq C_T,\label{A42}\\
&&\int^\infty_0 \|e^{rA}\sqrt{Q_2}\|^2_{HS}dr<\infty,\label{A43}
\end{eqnarray}
where $C_T>0$ is a constant depending on $T$ and $\|\cdot\|_{HS}$ is the norm of the Hilbert-Schmidit operator.
\end{conditionA}

\begin{conditionA}\label{A5}
Let $Q_i(t):=\int^t_0 e^{sA}Q_i e^{sA^{*}}ds$ ($i=1,2$) be two trace class operators. The well-defined bounded operator $\Lambda_i(t):=Q_i^{-1/2}(t)e^{t A}$ satisfies
\begin{align}
\int_0^{\infty}e^{-\lambda t}\|\Lambda_i(t)\|^{1+\kappa_1} dt<\infty,\quad \forall \lambda>0,\label{A50}
\end{align}
for some $\kappa_1\geq \max\{\alpha\wedge\beta\wedge\gamma, 1-\alpha\wedge(\beta\gamma)\}$, where $\|\cdot\|$ is the operator norm. Moreover, there exists $\kappa_2\in(0,1/2)$ such that
\begin{eqnarray}
\int^{\infty}_0 e^{-\lambda t}\|(-A)^{\kappa_2}\Lambda_1(t)\|dt<\infty,\quad \forall \lambda>0.\label{A51}
\end{eqnarray}
\end{conditionA}

\begin{conditionA}\label{A6} The smallest eigenvalue $\lambda_1$ of $A$ and the Lispchitz constants $L_{F}$ satisfies
$$
\lambda_{1}-L_{F}>0.
$$
\end{conditionA}

\begin{remark}
It seems a little strong that $F(x,y)$ is Lipschitz continuous with respect to $y$ uniformly for $x$ in \ref{A1}, which is needed to prove the H\"{o}lder continuous of the averaged coefficients $\bar{B}$ (see the detailed explantation in Remark \ref{R3.7} below). The assumptions \ref{A2}-\ref{A5} ensure the existence and uniqueness of the solution and the estimates of system \eref{Main equation}. The condition \ref{A6} is a strong dissipative condition, which can guarantee the existence and uniqueness of the invariant measure and the exponential ergodicity for the transition semigroup of the frozen equation.

\end{remark}
\begin{remark}
Note that the covariance operator $Q_1$ may not be a trace class operator.  The existence and uniqueness of the mild solution can be proved, but the It\^{o}'s formula can not be used directly in this case. However, the assumption of $Q_1(t)$ being a trace class operator is enough for applying the Zvonkin transform by the approximations, see e.g. \cite{DF}.
\end{remark}

\vspace{0.3cm}
Given $\alpha\in (0,1]$, denoted by $C^{\alpha}_b(H,H)$ the space of all bounded and H\"{o}lder continuous functions $G(x): H\rightarrow H$ with index $\alpha$ and norm
$$
\|G\|_{C^{\alpha}_b}:=\|G\|_{\infty}+\|G\|_{C^{\alpha}},
$$
where  $\|G\|_{\infty}:=\sup_{x\in H}|G(x)|$ and $\|G\|_{C^{\alpha}}:=\sup_{x\neq y\in H}\frac{|G(x)-G(y)|}{|x-y|^{\alpha}}$.

For any $s\in\RR$, we define
 $$H^s:=\mathscr{D}((-A)^{s/2}):=\left\{u=\sum_{k}u_ke_k: u_k=\langle u,e_k\rangle\in \mathbb{R},~\sum_k\lambda_k^{s}u_k^2<\infty\right\}$$
 and
 $$(-A)^{s/2}u:=\sum_k\lambda_k^{s/2} u_ke_k,~~u\in\mathscr{D}((-A)^{s/2}),$$
with the associated norm
\begin{eqnarray*}
\|u\|_{s}:=|(-A)^{s/2}u|=\sqrt{\sum_k\lambda_k^{s} u^2_k}.
\end{eqnarray*}
It is easy to see that $H^{0}=H$ and  $H^{-s}$ is the dual space of $H^s$. The dual action will also be denoted by $\langle\cdot,\cdot\rangle$ and $\|\cdot\|$ the operator norm without confusion.

Under the above conditions, one can check that for any $\theta>0$, there exists a constant $C_{\theta}>0$ such that

\begin{eqnarray}
&&|e^{tA}x|\leq e^{-\lambda_1 t}|x|, \quad x\in H, t\geq 0;\label{SP1}\\
&&\|e^{tA}x\|_\theta\leq C_{\theta} t^{-\frac{\theta}{2}}|x|, \quad x\in H, t>0;\label{SP2}\\
&&| e^{At}x-x | \leq C_{\theta}t^{\frac{\theta}{2}}\|x\|_{\theta},\quad x\in \mathscr{D}((-A)^{\frac{\theta}{2}}), t\geq 0.\label{SP3}\\
&&| e^{At}x-e^{As}x | \leq C_{\theta}\frac{(t-s)^{\theta}}{s^{\theta}}|x|,\quad x\in H, t>s> 0.\label{SP4}
\end{eqnarray}

\subsection{Main result}

The main result of this paper is as follows.
\begin{theorem}\label{main result 1}
Assume that the conditions \ref{A1}-\ref{A6} hold. Then for any $x, y\in H$, $p\geq1$ and $T>0$, we have
\begin{align}
\lim_{\vare\rightarrow 0}\mathbb{E} \left(\sup_{t\in[0,T]}|X_{t}^{\vare}-\bar{X}_{t}|^{p} \right)=0,\label{2.2}
\end{align}
where $\bar{X}_t$ is the solution of the following averaged equation:
\begin{equation}
d\bar{X}_{t}=A\bar{X}_{t}dt+\bar{B}(\bar{X}_{t})dt+\sqrt{Q_1}d W^1_{t},\quad
\bar{X}_{0}=x,\label{1.3}
\end{equation}
with $\bar{B}(x)=\int_{H}B(x,y)\mu^{x}(dy)$, and $\mu^{x}$ is the unique invariant measure of the transition semigroups for the  frozen equation
\begin{eqnarray}\label{FEQ1}
dY_{t}=[AY_{t}+F(x,Y_{t})]dt+\sqrt{Q_2}d{W}_{t}^{2},\quad Y_{0}=y.
\end{eqnarray}
\end{theorem}

\subsection{Idea of proof}

Since the coefficients of the system \eref{Main equation} are only H\"older continuous, the classical Khasminskii's time discreatization can't be used to prove our main result directly. Inspired from \cite{V0}, we shall use the Zvonkin's transformation to change the singular coefficients to regular ones. Such a technique is now well-known in the study of the well-posedness of S(P)DEs with singular coefficients. By a similar argument as in the \cite[Section 2]{DF}, we give a direct-viewing the idea of how to use the Zvonkin transformation, so we do not care about the rigor of the computations.

Consider the following PDE in $H$:
\begin{equation}\label{MPDE}
\lambda U(x)-\bar{\mathscr{L}}U(x)=\bar{B}(x),\quad x\in H,
\end{equation}
where $\lambda>0$ and $\bar{\mathscr{L}}$ is the infinitesimal generator of averaged equation, i.e.,
\begin{eqnarray}\label{gege}
\bar{\mathscr{L}}f(x)=\langle Ax, Df(x)\rangle+\langle \bar{B}(x), Df(x)\rangle+\tfrac{1}{2}\text{Tr}[D^2 f(x)Q_1].
\end{eqnarray}
If $U$ is a sufficiently regular solution, by It\^o's formula we have
\begin{eqnarray*}
dU(\bar{X}_t)=\!\!\!\!\!\!\!\!&&\lambda U(\bar X_t)dt-\bar B(\bar X_t)dt+DU(\bar X_t)\sqrt{Q_1} dW^1_t.
\end{eqnarray*}
As a result, we get
$$
\bar B(\bar X_t)dt=\lambda U(\bar X_t)dt-dU(\bar{X}_t)+DU(\bar X_t)\sqrt{Q_1} dW^1_t.
$$
We put this formula in equation (\ref{1.3}) and get
\begin{eqnarray*}
d\bar X_t=A\bar X_t dt+\lambda U(\bar X_t)dt-dU(\bar X_t)+(I+DU(\bar X_t))\sqrt{Q_1} dW^1_t,
\end{eqnarray*}
where $I$ is the identical operator. By variation of constant method and integration by parts formula, we get
\begin{eqnarray}
\bar X_t
=\!\!\!\!\!\!\!\!&&e^{tA}(x+U(x))+\int^t_0 e^{(t-s)A}\lambda U(\bar X_s)ds-U(\bar{X}_t)-\int^t_0 Ae^{(t-s)A}U(\bar X_s)ds\nonumber\\
&&+\int^t_0 e^{(t-s)A}(I+DU(\bar X_s))\sqrt{Q_1} dW^1_s.\label{FbarX}
\end{eqnarray}
Then by a similar argument, we also have
\begin{eqnarray}
X^{\vare}_t=\!\!\!\!\!\!\!\!&&e^{tA}(x+U(x))+\int^t_0 e^{(t-s)A}\lambda U(X^{\vare}_s)ds-U(X^{\vare}_t)-\int^t_0 Ae^{(t-s)A}U(X^{\vare}_s)ds\nonumber\\
&&\!\!+\int^t_0 e^{(t-s)A}(I+DU(X^{\vare}_s))\sqrt{Q_1}dW^1_s\nonumber\\
&&\!\!+\int^t_0 e^{(t-s)A}\langle (I+D U(X^{\vare}_s)), B(X^{\vare}_s, Y^{\vare}_s)-\bar{B}(X^{\vare}_s)\rangle ds.\label{FX}
\end{eqnarray}
Note that the non-regular drift $B$ has been removed in \eref{FbarX}. Although that the last  term in \eref{FX} is still non-regular, it is possible to be handled by a time discretization method and the exponential ergodicity of the transition semigroup of the frozen equation.

\section{Proof of main result} \label{Sec proof of result}

In this section, we are devoted to proving Theorem \ref{main result 1}. The proof consists of the following five subsections. In  Subsection 3.1, we show the well-posedness for system (1.1), denoted the unique solution by $(X^{\varepsilon}_t, Y^{\varepsilon}_t)$, and give  some a-priori estimates of the solution. In  Subsection 3.2, we study the frozen equation and its exponential ergodicity, which will be used in the final proof. The averaged equation and Zvonkin transformation are considered in  Subsection 3.3. In Subsection 3.4, we construct an auxiliary processes $(\hat{X}_{t}^{\varepsilon},\hat{Y}_{t}^{\varepsilon})\in H \times H$ and deduce an estimate of the difference process $Y^{\varepsilon}_t-\hat{Y}_{t}^{\varepsilon}$. Finally, we will give the detailed proof of Theorem \ref{main result 1}. We always assume \ref{A1}-\ref{A6} hold and the initial values $(x,y)\in H\times H$ are fixed in this section.

\subsection {Some a-priori estimates of \texorpdfstring{$(X^{\varepsilon}_t, Y^{\varepsilon}_t)$}{Lg}}

\begin{lemma} \label{PMY}
The system \eqref{Main equation} has a unique strong solution $(X^{\varepsilon},Y^{\varepsilon})$.
Moreover, for any $T>0$ and $p\geq1$, there exists a constant $C_{p,T}>0$ such that
\begin{align}
\sup_{\vare\in(0,1)}\mathbb{E}\left(\sup_{t\in[0,T]}|X_{t}^{\vare}|^{p}\right)\leq C_{p,T}\left(1+|x|^{p}\right)\label{F3.1}
\end{align}
and
\begin{align}
\sup_{\vare\in(0,1)}\sup_{t\geq 0}\mathbb{E}|Y_{t}^{\varepsilon}|^{p}\leq C_{p}\left(1+|y|^{p}\right).\label{ess}
\end{align}
\end{lemma}
\begin{proof}
Let $\mathcal{H}:=H\times H$ be the product Hilbert space. Rewrite the system \eref{Main equation} for $Z^{\varepsilon}_t=(X^{\varepsilon}_t,Y^{\varepsilon}_t)$ as
\begin{eqnarray*}
dZ^{\varepsilon}_t=\tilde{A}Z^{\varepsilon}_tdt+B^{\vare}(Z^{\varepsilon}_t)dt+\sqrt{Q}dW_t,\quad Z^{\varepsilon}_0=(x,y)\in \mathcal{H},
\end{eqnarray*}
where $W_t:=(W_t^{1},W_t^{2})$  is a $\mathcal{H}$-valued cylindrical-Wiener process, $Q$ is a bounded operator in  $\mathcal{H}$, which is denoted by $Qz=(Q_1x, Q_2y)$, for $z=(x,y)\in \mathcal{H}$, and
\begin{eqnarray*}
&&\tilde{A}Z^{\varepsilon}_t=\left(AX^{\varepsilon}_t,\frac{1}{\varepsilon}AY^{\varepsilon}_t\right),
\\&&B^{\vare}(Z^{\varepsilon}_t)=\left(B(X^{\varepsilon}_t,Y^{\varepsilon}_t),\frac{1}{\varepsilon}F(X^{\varepsilon}_t,Y^{\varepsilon}_t)\right).
\end{eqnarray*}
It is easy to see that $B^{\vare}$ is bounded and H\"{o}lder continuous with index $\alpha\wedge\beta\wedge\gamma$ in $\mathcal{H}$, i.e.,
$$
\|B^{\vare}(z_1)-B^{\vare}(z_2)\|_{\mathcal{H}}\leq C_{\vare}\|z_1-z_2\|^{\alpha\wedge\beta\wedge\gamma}_{\mathcal{H}},\quad z_1,z_2\in \mathcal{H}.
$$
Then under the assumptions \ref{A2}-\ref{A5}, the existence and uniqueness of strong solution in the mild sense for system \eqref{Main equation} follows by \cite[Theorem 7]{DF}.

Next, we intend to prove the a-priori estimates of the solution. By H\"{o}lder inequality, it suffices to prove (\ref{F3.1})  for large enough $p$.
Using the factorization method, for $\theta \in (0, 1)$ in \ref{A4}, we write
$$
W_{A}(t):=\int_{0}^{t}e^{(t-s)A}\sqrt{Q_1}dW^{1}_s=\frac{\sin( \pi\theta/2)}{\pi}\int^t_0 e^{(t-s)A}(t-s)^{\theta/2-1}Z_sds,
$$
where
$$
Z_s=\int^s_0 e^{(s-r)A}(s-r)^{-\theta/2}\sqrt{Q_1}dW^{1}_r.
$$

Choosing $p>1$ large enough such that $\frac{p(1-\theta/2)}{p-1}<1$, we get for any $t\in [0,T]$,
\begin{eqnarray*}
|W_{A}(t)|\leq C\left(\int^t_0 (t-s)^{-\frac{p(1-\theta/2)}{p-1}}ds\right)^{\frac{p-1}{p}}\|Z\|_{L^{p}(0,T; H)}
\leq C_{p}t^{\frac{\theta}{2}-\frac{1}{p}}\|Z\|_{L^{p}(0,T; H)},
\end{eqnarray*}
where $\|Z\|_{L^{p}(0,T; H)}:=\left(\int^T_0 |Z_t|^p dt\right)^{1/p}$. Then it implies
\begin{eqnarray}
\sup_{0 \leq t\leq T}|W_{A}(t)|^{p}\leq\!\!\!\!\!\!\!\!&&C_{p,T}\|Z\|^{p}_{L^{p}(0,T; H)}.\label{4.3}
\end{eqnarray}

Note that $Z_s\sim N(0, \tilde{Q}_s)$,  which is a Gaussian random variable with mean zero and covariance operator given by
$$
\tilde{Q}_s x=\int^s_0 r^{-\theta}e^{rA}Q_1 e^{rA^{*}}xdr.
$$
Then for any $p\geq1$,  $s\in[0, T]$, we use the condition \eref{A41} and follow the proof of \cite[Corollary 2.17]{Daz1}, it is easy to see that
\begin{align*}
\sup_{s\in[0, T]}\EE|Z_s|^{p}\leq &\   C_{p}\sup_{s\in[0, T]}[\text{Tr}(\tilde{Q}_s)]^{p/2}
= C_p\left(\int^T_0 r^{-\theta}\|e^{rA}\sqrt{Q_1}\|^2_{HS} dr\right)^{p/2} \leq C_{p, T},
\end{align*}
which yields
\begin{align}
\EE\left[\sup_{0 \leq t\leq T}|W_{A}(t)|^{p}\right]\leq C_{p,T}\EE\|Z\|^{p}_{L^{p}(0,T; H)}= &\  C_{p,T}\EE\int^T_0 |Z_s|^{p}ds\leq C_{p, T}.\label{CV}
\end{align}

By the boundedness of $B$ and \eref{CV},  it is easy to see
\begin{eqnarray*}
\EE\left(\sup_{t\in[0,T]}|X_{t}^{\varepsilon}|^{p}\right)\leq\!\!\!\!\!\!\!\!&&C_{p,T}(|x|^{p}+1)+C_{p,T}\EE\left[\sup_{t\in[0, T]}|W_A(t)|^{p}\right]\\
\leq\!\!\!\!\!\!\!\!&&C_{p,T}(|x|^{p}+1).
\end{eqnarray*}

Now, we proceed to show estimate (\ref{ess}). By the boundedness of $F$, property \eref{SP1} and Burkholder-Davis-Gundy's inequality, for any $t\geq 0$ we have
\begin{eqnarray*}
\EE|Y_{t}^{\varepsilon}|^{p}\leq\!\!\!\!\!\!\!\!&&C_p\left[e^{-t\lambda_1 p/\vare }|y|^{p}+\left(\int^t_0 \frac{1}{\vare}e^{-s \lambda_1/\vare}\|F\|_{\infty}ds\right)^{p}+\EE\left|\frac{1}{\sqrt{\vare}}\int^t_0 e^{(t-s)A/\vare}\sqrt{Q_2}d W^2_s\right|^{p}\right]\\
\leq\!\!\!\!\!\!\!\!&&C_{p}(1+|y|^{p})+\frac{C_p}{\vare^{p/2}}\left(\int^t_0 \left\|e^{(t-s)A/\vare}\sqrt{Q_2}\right\|^2_{HS}d s\right)^{p/2}\\
\leq\!\!\!\!\!\!\!\!&&C_{p}(1+|y|^{p})+C_p\left(\int^{\infty}_0 \left\|e^{rA}\sqrt{Q_2}\right\|^2_{HS}d r\right)^{p/2},
\end{eqnarray*}
which in turn implies the desired result by condition \eref{A43}.
The proof is complete.
\end{proof}

\begin{lemma} \label{SOX}
For any $t\in (0, T]$ and $p\geq 1$, there exists a constant $C_{p,T}>0$ such that
\begin{align}
\sup_{\varepsilon \in (0,1)} \mathbb{E}\|X_{t}^{\varepsilon}\|_{\theta}^{p}
\leq C_{p,T}t^{-\frac{\theta p}{2}}(|x|^{p}+1),\label{4.5}
\end{align}
where $\theta$ is given in \ref{A4}.
\end{lemma}

\begin{proof}
Recall that
\begin{align*}
X^{\varepsilon}_t=e^{tA}x+\int^t_0e^{(t-s)A}B(X^{\varepsilon}_s, Y^{\varepsilon}_s)ds+\int^t_0 e^{(t-s)A}\sqrt{Q_1}dW^{1}_s.
\end{align*}
For the first term, we have by \eref{SP2} that
\begin{eqnarray}
\|e^{At}x\|_{\theta}^{p}\leq C t^{-\frac{\theta p}{2}}|x|^{p}.\label{H1}
\end{eqnarray}
For the second term, by \eref{SP2} and the boundedness of $B$, we can get
\begin{eqnarray}
\mathbb{E}\Big\|\int^t_0e^{(t-s)A}B(X^{\varepsilon}_s, Y^{\varepsilon}_s)ds\Big\|_{\theta}^{p}\leq C\Big[
\int^t_0 (t-s)^{-\frac{\theta}{2}}ds\Big]^{p}\leq C_{p,T}. \label{H2}
\end{eqnarray}
For the third term, by Burkholder-Davis-Gundy's inequality and condition \eref{A42}, we have for any $t\in [0, T]$,
\begin{eqnarray}
\mathbb{E}\Big\|\int^t_0 e^{(t-s)A}\sqrt{Q_1}dW^{1}_s\Big\|_{\theta}^{p}\leq\!\!\!\!\!\!\!\!&& C_p\left(\int^t_0 \|(-A)^{\theta/2}e^{(t-s)A}\sqrt{Q_1}\|^2_{HS}ds\right)^{p/2}\nonumber\\
\leq \!\!\!\!\!\!\!\!&&C_p\left( \int^T_0 \|(-A)^{\theta/2}e^{sA}\sqrt{Q_1}\|^2_{HS}ds\right)^{p/2}\leq C_{p,T} .\label{H3}
\end{eqnarray}
Hence, the proof is completed by combining $\eref{H1}$-$\eref{H3}$.
\end{proof}

Usually, the H\"{o}lder continuity of $X_{t}^{\varepsilon}$ in time plays an important role in the method of time discretization (see \cite[Proposition 4.4]{C1}, \cite[Lemma 3,4]{DSXZ} and \cite[Proposition 9]{GP}), then the initial value $x\in H^{\theta}$ will be assumed for some $\theta>0$. However inspired from \cite{LRSX2}, studying the H\"{o}lder continuity can be replaced by studying the integral of the time increment of $X_{t}^{\varepsilon}$, which is weaker than the H\"{o}lder continuity but  enough for our purpose, and it only needs initial value $x\in H$ for advantage.

\begin{lemma} \label{COX}
For any $T>0$ and $p\geq 1$, there exists a constant $C_{p,T}>0$ such that for any $\vare\in(0,1)$ and $\delta>0$ small enough,
\begin{align}
\mathbb{E}\left[\int^{T}_0|X_{t}^{\varepsilon}-X_{t(\delta)}^{\varepsilon}|^2 dt\right]\leq C_{T}\delta^{\theta}(1+|x|^2),\label{F3.7}
\end{align}
where $t(\delta):=[\frac{t}{\delta}]\delta$ and $[s]$ denotes the integer part of $s$ and  $\theta$ is given in \ref{A4}.
\end{lemma}

\begin{proof}
It is easy to see that
\begin{eqnarray}
&&\mathbb{E}\left[\int^{T}_0|X_{t}^{\varepsilon}-X_{t(\delta)}^{\varepsilon}|^2 dt\right]\nonumber\\
=\!\!\!\!\!\!\!\!&& \mathbb{E}\left(\int^{\delta}_0|X_{t}^{\varepsilon}-x|^2 dt\right)+\mathbb{E}\left[\int^{T}_{\delta}|X_{t}^{\varepsilon}-X_{t(\delta)}^{\varepsilon}|^2 dt\right]\nonumber\\
\leq\!\!\!\!\!\!\!\!&& C_{T}\delta(1+|x|^2) +2\mathbb{E}\left(\int^{T}_{\delta}|X_{t}^{\varepsilon}-X_{t-\delta}^{\varepsilon}|^2 dt\right)+2\mathbb{E}\left(\int^{T}_{\delta}|X_{t(\delta)}^{\varepsilon}-X_{t-\delta}^{\varepsilon}|^2 dt\right),\label{F3.8}
\end{eqnarray}
where we use \eref{F3.1} in the last inequality. After simple calculations, we have
\begin{eqnarray}
X_{t}^{\varepsilon}-X_{t-\delta}^{\varepsilon}=\!\!\!\!\!\!\!\!&&(e^{A\delta}-I)X_{t-\delta}^{\varepsilon}+\int_{t-\delta}^{t}e^{(t-s)A}B(X^{\varepsilon}_s, Y^{\varepsilon}_s)ds+\int_{t-\delta}^{t}e^{(t-s)A}\sqrt{Q_1}dW^{1}_s  \nonumber\\
:=\!\!\!\!\!\!\!\!&&I_{1}(t)+I_{2}(t)+I_{3}(t). \label{REGX}
\end{eqnarray}

For the first term $I_1(t)$, by property \eref{SP3} and Lemma \ref{SOX}, there exists a constant $C>0$ such that
\begin{eqnarray}  \label{REGX1}
\mathbb{E}\left(\int^{T}_{\delta}|I_{1}(t)|^2 dt\right)\leq\!\!\!\!\!\!\!\!&& C\mathbb{E}\int^{T}_{\delta}\delta^{\theta}
\|X_{t-\delta}^{\varepsilon}\|^2_{\theta} dt\nonumber\\
\leq\!\!\!\!\!\!\!\!&&C\delta^{\theta}\int^{T}_{\delta} \left[C(t-\delta)^{-\theta}|x|^{2}+C_T\right]dt
\leq C_{T}\delta^{\theta}(1+|x|^2).
\end{eqnarray}

For the term $I_{2}(t)$, by the boundedness of $B$, it is easy to see
\begin{eqnarray}\label{REGX2}
\mathbb{E}\left(\int^{T}_{\delta}|I_{2}(t)|^2 dt\right)\leq C_{T}\delta^2.
\end{eqnarray}

For the term $I_{3}(t)$, by condition \eref{A41}, we get
\begin{eqnarray}  \label{REGX3}
\mathbb{E}\left(\int^{T}_{\delta}|I_{3}(t)|^2dt\right)\leq\!\!\!\!\!\!\!\!&&C\int^{T}_{\delta}\int_{t-\delta} ^{t}\|e^{(t-s)A}\sqrt{Q_1}\|^2_{HS} dsdt\nonumber\\
\leq\!\!\!\!\!\!\!\!&&C\delta^{\theta}\int^{T}_{\delta}\int_{0} ^{\delta}s^{-\theta}\|e^{sA}\sqrt{Q_1}\|^2_{HS} dsdt
\leq C_T\delta^{\theta}.
\end{eqnarray}

Combining estimates \eref{REGX}-\eref{REGX3}, we get that
\begin{eqnarray}
\mathbb{E}\left(\int^{T}_{\delta}|X_{t}^{\varepsilon}-X_{t-\delta}^{\varepsilon}|^2dt\right)\leq\!\!\!\!\!\!\!\!&&C_{T}\delta^{\theta}(1+|x|^2). \label{F3.13}
\end{eqnarray}
By a similar argument as above, we have
\begin{eqnarray}
\mathbb{E}\left(\int^{T}_{\delta}|X_{t(\delta)}^{\varepsilon}-X_{t-\delta}^{\varepsilon}|^2dt\right)\leq\!\!\!\!\!\!\!\!&&C_{T}\delta^{\theta}(1+|x|^2). \label{F3.14}
\end{eqnarray}
Hence, \eref{F3.8}, \eref{F3.13} and \eref{F3.14} imply \eref{F3.7} holds. The proof is complete.
\end{proof}

\vskip 0.3cm

\subsection{The frozen  equation and exponential ergodicity}

Recall that the frozen equation is given by (\ref{FEQ1}). Since $F(x,\cdot)$ is Lipshcitz continuous for any fixed $x\in {H}$, it is easy to prove that the equation $\eref{FEQ1}$ has a unique mild solution, denoted by $Y_{t}^{x,y}$, for any fixed $x\in {H}$ and any initial data $y\in H$. By almost the same steps in proving \eref{ess}, it is easy to see that $\sup_{t\geq 0}\EE|Y_{t}^{x,y}|^2\leq C(1+|y|^2)$.

For any fixed $x\in {H}$, let $P^{x}_t$ be the transition semigroup of $Y_{t}^{x,y}$,
that is, for any bounded measurable function $\varphi$ on $H$,
\begin{eqnarray*}
P^x_t \varphi(y)= \mathbb{E} \left[\varphi\left(Y_{t}^{x,y}\right)\right], \quad y \in H,\ \ t>0.
\end{eqnarray*}
By the condition \ref{A6} and \cite[Theorem 4.3.9]{LR}, we can prove that there is a unique invariant measure for $P^x_t$, denoted by $\mu^x$. Before proving the asymptotical behavior of $P^x_t$, we first give the following Lemma.

\begin{lemma} \label{L3.4} There exists a constant $C>0$ such that, for all $x_1,x_2, y_1,y_2\in H$ and $t\geq 0$,
\begin{eqnarray*}
|Y^{x_1,y_1}_t-Y^{x_2,y_2}_t|^2\leq C|x_1-x_2|^{2\gamma}+e^{-(\lambda_1-L_F) t}|y_1-y_2|^{2}.
\end{eqnarray*}
\end{lemma}
\begin{proof}
Note that
\begin{eqnarray*}
d(Y^{x_1,y_1}_t-Y^{x_2,y_2}_t)=\!\!\!\!\!\!\!\!&&A(Y^{x_1,y_1}_t-Y^{x_2,y_2}_t)dt+\left[F(x_1, Y^{x_1,y_1}_t)-F(x_2, Y^{x_2,y_2}_t)\right]dt.
\end{eqnarray*}
Then by condition \ref{A6} and Young's inequality, we have
\begin{eqnarray*}
&&\frac{d}{dt}|Y^{x_1,y_1}_t-Y^{x_2,y_2}_t|^2\\
=\!\!\!\!\!\!\!\!&&-2\|Y^{x_1,y_1}_t-Y^{x_2,y_2}_t\|_1^2+2\langle F(x_1, Y^{x_1,y_1}_t)-F(x_2, Y^{x_2,y_2}_t), Y^{x_1,y_1}_t-Y^{x_2,y_2}_t\rangle\\
\leq\!\!\!\!\!\!\!\!&&-2\lambda_1|Y^{x_1,y_1}_t-Y^{x_2,y_2}_t|^2+2L_F|Y^{x_1,y_1}_t-Y^{x_2,y_2}_t|^2+C|x_1-x_2|^{\gamma}|Y^{x_1,y_1}_t-Y^{x_2,y_2}_t|\\
\leq\!\!\!\!\!\!\!\!&&-(\lambda_1-L_F)|Y^{x_1,y_1}_t-Y^{x_2,y_2}_t|^2+C|x_1-x_2|^{2\gamma}.
\end{eqnarray*}
Hence, the comparison theorem yields that
\begin{eqnarray*}
|Y^{x_1,y_1}_t-Y^{x_2,y_2}_t|^2\leq\!\!\!\!\!\!\!\!&&e^{-(\lambda_1-L_F)t}|y_1-y_2|^2+C\int^t_0 e^{-(\lambda_1-L_F)(t-s)}ds |x_1-x_2|^{2\gamma}\\
\leq\!\!\!\!\!\!\!\!&&C|x_1-x_2|^{2\gamma}+e^{-(\lambda_1-L_F)t}|y_1-y_2|^{2}.
\end{eqnarray*}
The proof is complete.
\end{proof}

Now, we give a position to prove the following exponential behavior of $P^x_t$.

\begin{proposition}\label{Ergodicity}
There exists $C>0$  such that for any H\"{o}lder continuous function $\varphi:H\rightarrow H$ with index $\eta\in (0,1]$ and $x,y\in H$,
\begin{equation}
\Big|P^x_t\varphi(y)-\int_{H}\varphi(z)\mu^x(dz)\Big|\leq C(1+|y|^{\eta})\|\varphi\|_{C^{\eta}}e^{-\frac{(\lambda_1-L_F)\eta t}{2}}.
\end{equation}
\end{proposition}
\begin{proof}
For any H\"{o}lder continuous function $\varphi:H\rightarrow H$ with index $\eta\in (0,1]$, by the definition of invariant measure and Lemma \ref{L3.4}, we have
\begin{eqnarray*}
\left|P^{x}_t\varphi(y)-\int_{H}\varphi(z)\mu^{x}(dz)\right|\leq\!\!\!\!\!\!\!\!&&\int_{H}\left|\mathbb{E}\varphi(Y^{x,y}_t)-\EE\varphi(Y^{x,z}_t)\right|\mu^{x}(dz)\\
\leq\!\!\!\!\!\!\!\!&&\|\varphi\|_{C^{\eta}}\int_{H}\mathbb{E}\left|Y^{x,y}_t-Y^{x,z}_t\right|^{\eta}\mu^{x}(dz)\\
\leq\!\!\!\!\!\!\!\!&&\|\varphi\|_{C^{\eta}}\int_{H}e^{-\frac{(\lambda_1-L_F)\eta t}{2} }|y-z|^{\eta}\mu^{x}(dz)\\
\leq\!\!\!\!\!\!\!\!&&C(1+|y|^{\eta})\|\varphi\|_{C^{\eta}}e^{-\frac{(\lambda_1-L_F)\eta t}{2} },
\end{eqnarray*}
where the last inequality comes from $\int_{H}|z|^{\eta}\mu^{x}(dz)\leq \sup_{t\geq 0}\EE|Y^{x,0}|^{\eta}\leq C$. The proof is completed.
\end{proof}

\subsection{Averaged equation and Zvonkin transformation}
In this subsection, we first recall the averaged equation, i.e.,
\begin{equation}
\displaystyle d\bar{X}_{t}=A\bar{X}_{t}dt+\bar{B}(\bar{X}_{t})dt+\sqrt{Q_1}dW^{1}_t,\quad
\bar{X}_{0}=x \label{3.1}
\end{equation}
with
\begin{align*}
\bar{B}(x)=\int_{H}B(x,y)\mu^{x}(dy),
\end{align*}
where $\mu^{x}$ is the unique invariant measure of the transition semigroup for  equation $\eref{FEQ1}$.

\vskip 0.1cm

\begin{lemma}\label{L3.8}
For any $x\in H$, the equation $\eref{3.1}$ has a unique strong solution $\bar{X}$.
Moreover, for any $T>0$ and $p\geq 1$, there exists a positive constant $C_{p,T}$ such that
\begin{align}
\mathbb{E}\left(\sup_{t\in[0,T]}|\bar{X}_{t}|^{p}\right)\leq C_{p,T}(1+|x|^{p}).\label{EbarX}
\end{align}
\end{lemma}

\begin{proof}
Obviously, $\bar{B}$ is bounded due to the boundedness of $B$. Next, we can check that $\bar{B}$ satisfies the following:
\begin{eqnarray}
|\bar{B}(x_1)-\bar{B}(x_2)|\leq C|x_1-x_2|^{\alpha\wedge(\beta \gamma)},\quad  x_1,x_2\in H.\label{HObarB}
\end{eqnarray}
In fact, Proposition \ref{Ergodicity} and Lemma \ref{L3.4} imply that
\begin{eqnarray*}
|\bar{B}(x_1)-\bar{B}(x_2)|=\!\!\!\!\!\!\!\!&&\left|\int_{H} B(x_1,z)\mu^{x_1}(dz)-\int_{H} B(x_2,z)\mu^{x_2}(dz)\right|\\
\leq\!\!\!\!\!\!\!\!&&\left|\int_{H} B(x_1,z)\mu^{x_1}(dz)-{\EE} B(x_1, Y^{x_1,0}_t)\right|\\
&&+\left|  \EE B(x_2, Y^{x_2,0}_t)-\int_{H} B(x_2,z)\mu^{x_2}(dz)\right|\\
&&+ {\EE} \left|B(x_1, Y^{x_1,0}_t)-B(x_2, Y^{x_1,0}_t)\right|+ {\EE} \left|B(x_2, Y^{x_1,0}_t)-B(x_2, Y^{x_2,0}_t)\right|\\
\leq\!\!\!\!\!\!\!\!&&C e^{-\frac{(\lambda_1-L_F)\beta t}{2}}+C{\EE} \left|B(x_1, Y^{x_1,0}_t)-B(x_2, Y^{x_1,0}_t)\right|^{\frac{\alpha\wedge(\beta\gamma)}{\alpha}}\\
&&+ C{\EE} \left|B(x_2, Y^{x_1,0}_t)-B(x_2, Y^{x_2,0}_t)\right|^{\frac{\alpha\wedge(\beta\gamma)}{\beta\gamma}}\\
\leq\!\!\!\!\!\!\!\!&&C e^{-\frac{(\lambda_1-L_F)\beta t}{2}}+C|x_1-x_2|^{\alpha\wedge(\beta\gamma)},
\end{eqnarray*}
where the second inequality comes from the boundedness of $B$. Hence, letting $t\rightarrow \infty$, we obtain (\ref{HObarB}).
Then Eq.(\ref{3.1}) has a unique solution by \cite[Theorem 7]{DF} under the assumptions \ref{A2}-\ref{A5}. Moreover, (\ref{EbarX}) can be easily obtained by the same argument as in Lemma \ref{PMY}.
\end{proof}

\begin{remark}\label{R3.7}
It is worthy to point out that the expected H\"{o}lder index of $\bar{B}$ should be $\alpha\wedge\gamma$. But due to the technique used here, we only obtain that the index is $\alpha\wedge(\beta\gamma)$. By the way, if $F(x,y)$ is assumed H\"{o}lder continuous with respect to $y$ only, it will be difficult  to prove that $\bar{B}$ is still H\"{o}lder continuous. In fact, the key of proof for Lemma \ref{L3.8} is the estimate $\EE|Y^{x_1,y}_t-Y^{x_2,y}_t|^2\leq C|x_1-x_2|^{2\gamma}$. Where $C$ must be independent of $t$. However, if $F(x,y)$ is H\"{o}lder continuous with respect to $y$, then the method of proof used in Lemma \ref{L3.8} does not work any more. Although we can use the Zvonkin's transformation to change the H\"{o}lder continuous to Lipschitz continuous, and the estimate $\EE|Y^{x_1,y}_t-Y^{x_2,y}_t|^2\leq C_t|x_1-x_2|^{2\gamma}$, where $C_t$ depends on $t$. This is not good enough for our purpose since $C_t$ may tend to infinity as $t\uparrow \infty$.
\end{remark}

Now, we study the Zvonkin transformation of the averaged equation. Firstly, we introduce the following PDE:
\begin{equation}\label{PDE1}
\lambda U(x)-\bar{\mathscr{L}}U(x)=G(x),\quad x\in H,
\end{equation}
where $\lambda>0$, $\bar{\mathscr{L}}$ is given by (\ref{gege}) and $G: H\rightarrow H$ is measurable.

Let $C^{2}_b(H,H)$ be the space of the functions (from $H$ to $H$) which are bounded and twice differentiable, with first and second order bounded derivatives. If $U\in C^{2}_b(H,H)$, its norm is defined by
$$
\|U\|_{C^{2}_b}:=\|U\|_{\infty}+\|DU\|_{\infty}+\|D^2U\|_{\infty}.
$$

We have the following result.

\begin{lemma}\label{ppp}
For every $G\in C_b^{\alpha\wedge(\beta\gamma)}(H,H)$, there exists a function $U\in C^2_b(H,H)$ satisfying the following integral equation:
\begin{align}
U(x)=\int_0^\infty\!\e^{-\lambda t}T_t\Big(\langle \bar{B}, DU\rangle+G\Big)(x)dt.   \label{inte}
\end{align}
where $T_t f(x):=\EE[f(Z^{x}_t)]$ for any $f\in B_b(H,H)$, $Z^{x}_t$ is the unique solution of following equation
\begin{eqnarray}
dZ^{x}_t=AZ^{x}_tdt+\sqrt{Q_1}dW^1_t,\quad Z^{x}_0=x\in H.\label{OUEq}
\end{eqnarray}
Moreover, $U$ also solves equation (\ref{PDE1}) and the following estimates hold:
\begin{eqnarray}
&&\|U\|_{C^2_b}\leq D_{\lambda}\|G\|_{C^{\alpha\wedge(\beta\gamma)}_b};\label{BU}\\
&&\|(-A)^{\kappa_2}DU\|_{\infty}\leq C\|G\|_{C^{\alpha\wedge(\beta\gamma)}_b},\label{BU2}
\end{eqnarray}
where $C,D_{\lambda}>0$ are two constants with $\lim_{\lambda\to \infty}D_{\lambda}=0$.
\end{lemma}

\begin{proof}
We construct a solution to (\ref{inte}) via the Picard's iteration argument.
Set $U_0\equiv0$ and for $n\in \mathbb{N}$, define $U_n$ recursively by
\begin{eqnarray}
U_n(x):=\int_0^\infty\!\e^{-\lambda t}T_t\Big(\langle \bar{B},DU_{n-1}\rangle+G\Big)(x)dt.
\end{eqnarray}
Refer to \cite[Theorem 4]{DF}, we have for any $f\in C^{\theta}_b(H,H)$ with $\theta\in (0,1]$,
\begin{eqnarray}
\|D T_t f\|_{\infty}\leq C\|\Lambda_t\|\|f\|_{\infty},\quad \|D^2 T_t f\|_{\infty}\leq C\|\Lambda_t\|^{2-\theta}\|f\|_{C^{\theta}_b}.\label{OU}
\end{eqnarray}
Then it is easy to check that $U_1\in C^1_b(H, H)$, and $U_2$ is thus well defined, and so on. We claim that $U_1\in C_b^{2}(H, H)$. In fact, thanks to condition \ref{A50} and using (\ref{OU}), we have
\begin{align*}
\|D^2U_1\|_{\infty}&\leq \int_0^\infty\!\e^{-\lambda t}\|D^2 T_tG\|_{\infty}dt\\
&\leq C\|G\|_{C^{\alpha\wedge(\beta\gamma)}_b}\int_0^\infty\!\e^{-\lambda t}\|\Lambda_t\|^{2-\alpha\wedge(\beta\gamma)}dt:=C D_{\lambda}\|G\|_{C^{\alpha\wedge(\beta\gamma)}_b},
\end{align*}
where $D_\lambda:=\int_0^\infty\!\e^{-\lambda t}\|\Lambda_t\|^{2-\alpha\wedge(\beta\gamma)}dt$, and by dominate convergence theorem it holds that $\lim_{\lambda\rightarrow +\infty}D_\lambda=0$.
As a result, note that $\bar{B}\in C_b^{\alpha\wedge(\beta\gamma)}(H,H)$, it holds
$$
\langle \bar{B}, DU_1\rangle \in C_b^{\alpha\wedge(\beta\gamma)}(H,H).
$$
Repeating the above argument, we can get for every $n\in \mathbb{N}$,
$$
U_n\in C_b^{2}(H, H).
$$
Moreover, for any $n>m$
\begin{align*}
U_n(x)-U_m(x)&=\int_0^\infty\!\e^{-\lambda t}T_t\Big(\langle \bar{B}, DU_{n-1}-DU_{m-1}\rangle\Big)(x)dt.
\end{align*}
By (\ref{OU}), we further have that
\begin{align*}
\|U_n-U_m\|_{C^2_b}&\leq \int_0^\infty\!\e^{-\lambda t}\Big\|T_t\Big(\langle \bar{B}, DU_{n-1}-DU_{m-1}\rangle\Big)\Big\|_{C^2_b}d t\\
&\leq C\!\!\int_0^\infty\!\e^{-\lambda t}\|\Lambda_t\|^{2-\alpha\wedge(\beta\gamma)} dt\cdot\|\langle \bar{B}, DU_{n-1}-DU_{m-1}\rangle\|_{C^{\alpha\wedge(\beta\gamma)}_b}\\
&\leq CD_{\lambda}\|\bar{B}\|_{C^{\alpha\wedge(\beta\gamma)}_b}\cdot\|DU_{n-1}-DU_{m-1}\|_{C^{\alpha\wedge(\beta\gamma)}_b}\\
&\leq CD_{\lambda}\|\bar{B}\|_{C^{\alpha\wedge(\beta\gamma)}_b}\cdot\|U_{n-1}-U_{m-1}\|_{C^2_b}.
\end{align*}
This means that for $\lambda$ large enough, $U_n$ is Cauchy sequence in $C_b^{2}(H, H)$. Thus, there exists a limit function $U\in C_b^{2}(H, H)$ satisfying (\ref{inte}). The assertion that $U$ solves (\ref{PDE1}) follows by integration by parts.

Now we show the estimates (\ref{BU}) and (\ref{BU2}). By (\ref{OU}) and \ref{A50} again, we have
\begin{align*}
\|U\|_{C^2_b}&\leq \int_0^\infty\!\e^{-\lambda t}\Big\|T_t\Big(\langle \bar{B}, DU\rangle+G\Big)\Big\|_{C^2_b}dt\\
&\leq C\int_0^\infty\!\e^{-\lambda t}\|\Lambda_t\|^{2-\alpha\wedge(\beta\gamma)}dt\cdot\|\langle \bar{B}, DU\rangle+G\|_{C^{\alpha\wedge(\beta\gamma)}_b}\\
&\leq C D_{\lambda}\Big(\|G\|_{C^{\alpha\wedge(\beta\gamma)}}+\|\bar{B}\|_{C^{\alpha\wedge(\beta\gamma)}}\cdot\|U\|_{C^2_b}\Big).
\end{align*}
Taking $\lambda$ large enough such that $C D_\lambda\|\bar{B}\|_{C^{\alpha\wedge(\beta\gamma)}}\leq \frac{1}{2}$, we get the desired estimate (\ref{BU}).

Note that for any $\kappa_2\in(0,1/2)$, 
we can prove that
$$
\|(-A)^{\kappa_2}D T_t f\|\leq C\|(-A)^{\kappa_2}\Lambda_t\|\|f\|_{\infty}.
$$
Then by condition \ref{A51}, it is easy to see that
\begin{align*}
\|(-A)^{\kappa_2}DU\|_{\infty}&= \int_0^\infty\!\e^{-\lambda t}\|(-A)^{\kappa_2}DT_t\Big(\langle \bar{B}, DU\rangle+G\Big)\|_{\infty}dt\\
&\leq\int_0^\infty\!\e^{-\lambda t}\|(-A)^{\kappa_2}\Lambda_t\|dt\cdot\|\langle \bar{B}, DU\rangle+G\|_{\infty}\leq C\|G\|_{C^{\alpha\wedge(\beta\gamma)}_b}.
\end{align*}
The whole proof is finished.
\end{proof}

Now, we prove the following Zvonkin's transformation.

\begin{lemma}\label{ZT}
Let $\bar X_t$ be the solution of equation \eref{3.1}. Let $U$ be the solution of Eq. \eref{MPDE}. Then the formulas \eref{FbarX} and \eref{FX} hold.
\end{lemma}
\begin{proof} Inspired from \cite{DF}, the idea of proof is the one given in subsection 2.1. The only point is the application of It\^o's formula. On one hand, duo to $\bar{B}\in C^{\alpha\wedge (\beta\gamma)}_b(H,H)$, Eq. \eref{MPDE} has a solution $U\in C^2_{b}(H,H)$ which satisfies \eref{BU} and \eref{BU2} by Lemma \ref{ppp}. On the other hand,  we introduce the approximations:
\begin{eqnarray*}
d\bar X^{m,n}_{t}=[A_m\bar X^{m,n}_{t}+\bar B(\bar X^{m,n}_{t})]dt+\sqrt{Q_1}\Pi_n dW_{t}^{1},\quad\bar X^{m,n}_{0}=x,
\label{FFEQ1}
\end{eqnarray*}
where $A_m$ are the Yosida approximations of $A$ and $\Pi_n$ is the orthogonal projection of $H$ onto $\text{span}\{e_1,..., e_n\}$. Then the argument in Subsection 2.1 can be done on these approximations and then one can pass to the limit in both sides. We omit the details which are classical.
\end{proof}

\subsection{Construction of auxiliary processes}

Following the idea in \cite{K1}, we introduce an auxiliary process $(\hat{X}_{t}^{\varepsilon},\hat{Y}_{t}^{\varepsilon})\in{H}\times H$ and divide $[0,T]$ into intervals of size $\delta$, where $\delta$ is a fixed positive number depending on $\vare$ and will be chosen later.

We construct a process $\hat{Y}_{t}^{\varepsilon}$, with $\hat{Y}_{0}^{\varepsilon}=Y^{\varepsilon}_{0}=y$, and for any $k\in \mathbb{N}$ and $t\in[k\delta,\min((k+1)\delta,T)]$,
\begin{eqnarray}
\hat{Y}_{t}^{\varepsilon}=\hat Y_{k\delta}^{\varepsilon}+\frac{1}{\varepsilon}\int_{k\delta}^{t}A\hat{Y}_{s}^{\varepsilon}ds+\frac{1}{\varepsilon}\int_{k\delta}^{t}
F(X_{k\delta}^{\varepsilon},\hat{Y}_{s}^{\varepsilon})ds+\frac{1}{\sqrt{\varepsilon}}\int_{k\delta}^{t}\sqrt{Q_2}dW^{2}_s,\label{4.6a}
\end{eqnarray}
which is equivalent to
$$
d\hat{Y}_{t}^{\vare}=\frac{1}{\vare}\left[A\hat{Y}_{t}^{\vare}+F\left(X^{\vare}_{t(\delta)},\hat{Y}_{t}^{\vare}\right)\right]dt+\frac{1}{\sqrt{\vare}}\sqrt{Q_2}dW^{2}_t,\quad \hat{Y}_{0}^{\vare}=y.
$$
We also construct another auxiliary process $\hat{X}_{t}^{\vare}\in H$ by
\begin{eqnarray}
\hat X^{\vare}_t=\!\!\!\!\!\!\!\!&&e^{tA}(x+U(x))+\int^t_0 e^{(t-s)A}\lambda U(X^{\vare}_s)ds-U(X^{\vare}_t)-\int^t_0 Ae^{(t-s)A}U(X^{\vare}_s)ds\nonumber\\
&&\!\!+\int^t_0 e^{(t-s)A}\sqrt{Q_1}dW^1_s+\int^t_0 e^{(t-s)A}DU(X^{\vare}_s)\sqrt{Q_1}dW^1_s\nonumber\\
&&\!\!+\int^t_0 e^{(t-s(\delta))A}\langle D U(X^{\vare}_{s(\delta)})+I, B(X^{\vare}_{s(\delta)}, \hat Y^{\vare}_s)-\bar{B}(X^{\vare}_{s(\delta)})\rangle ds.\label{AMX}
\end{eqnarray}

\begin{lemma} \label{DEY}
For any $T>0$, there exists a constant $C_{T}>0$ such that for any $x, y\in H$ and $\vare\in(0,1)$,
\begin{eqnarray}
\mathbb{E}\left(\int_0^{T}|Y_{t}^{\varepsilon}-\hat{Y}_{t}^{\varepsilon}|^2dt\right)\leq C_{T}(|x|^{2\gamma}+1)\delta^{\theta\gamma}. \label{3.14}
\end{eqnarray}
\end{lemma}

\begin{proof}
Let $\rho^{\vare}_t:=Y_{t}^{\varepsilon}-\hat{Y}_{t}^{\varepsilon}$
Then, it is easy to see that $\rho^{\vare}_t$ satisfies the following equation:
\begin{equation*} d\rho^{\vare}_t=\frac{1}{\varepsilon}\left[A\rho^{\vare}_t+F(X_t^\varepsilon,Y_t^\varepsilon)-F(X_{t(\delta)}^\varepsilon,\hat{Y}_t^\varepsilon)\right]dt,\quad \rho^{\vare}_0=0.
\end{equation*}
Then by condition \ref{A6}, we get
\begin{eqnarray*}
\frac{d}{dt}\EE|\rho^{\vare}_t|^2
=\!\!\!\!\!\!\!\!&&-\frac{2}{\varepsilon}\EE\|\rho^{\vare}_t\|_1^2+\frac{2}{\varepsilon}\EE\langle F(X_t^\varepsilon,Y_t^\varepsilon)-F(X_{t(\delta)}^\varepsilon,\hat{Y}_t^\varepsilon),\rho^{\vare}_t\rangle \nonumber\\
\leq\!\!\!\!\!\!\!\!&&-\frac{2\lambda_1}{\varepsilon}\EE|\rho^{\vare}_t|^2+\frac{C}{\varepsilon}\EE\left(|X_t^\varepsilon-X_{t(\delta)}^\varepsilon|^{\gamma}\cdot|\rho^{\vare}_t|\right)+\frac{2L_F}{\varepsilon}\EE|\rho^{\vare}_t|^2  \nonumber\\
\leq\!\!\!\!\!\!\!\!&&-\frac{\lambda_1-L_F}{\varepsilon}\EE|\rho^{\vare}_t|^2+\frac{C}{\varepsilon}\EE|X_t^\varepsilon-X_{t(\delta)}^\varepsilon|^{2\gamma}.\nonumber
\end{eqnarray*}
Therefore, by the comparison theorem we have
\begin{eqnarray*}
\EE|\rho^{\vare}_t|^2\leq\frac{C}{\varepsilon}\int_0^te^{-\frac{(\lambda_1-L_F)(t-s)}{\vare}}\EE|X_s^\varepsilon-X_{s(\delta)}^\varepsilon|^{2\gamma}ds.
\end{eqnarray*}
Then by Fubini's theorem and Lemma \ref{COX}, we get that for any $T>0$,
\begin{eqnarray*}
\EE\left(\int_0^T|\rho^{\vare}_t|^2dt\right)\leq\!\!\!\!\!\!\!\!&& \frac{C}{\varepsilon}\int_0^T\int^t_0e^{-\frac{(\lambda_1-L_F)(t-s)}{\vare}}\EE|X_s^\varepsilon-X_{s(\delta)}^\varepsilon|^{2\gamma}dsdt\nonumber\\
=\!\!\!\!\!\!\!\!&&  \frac{C}{\varepsilon}\EE\left[\int_0^T|X_s^\varepsilon-X_{s(\delta)}^\varepsilon|^{2\gamma}\left(\int^T_s e^{-\frac{(\lambda_1-L_F)(t-s)}{\vare}}dt\right)ds\right]\nonumber\\
\leq\!\!\!\!\!\!\!\!&& C_T\left(\EE\int_0^T|X_s^\varepsilon-X_{s(\delta)}^\varepsilon|^2 ds\right)^{\gamma}\leq C_{T}(|x|^{2\gamma}+1)\delta^{\theta\gamma}.
\end{eqnarray*}
The proof is complete.
\end{proof}

\subsection{Proof of theorem \ref{main result 1}}
This section is devoted to giving the proof of our main result. We first give the estimate for the difference process $X_{t}^{\vare}-\hat{X}_{t}^{\vare}$.

\begin{lemma} \label{DEX} For any $T>0$ and $p\geq 2$, there exists $C_{p,T}>0$ such that
\begin{eqnarray*}
\mathbb{E}\Big(\sup_{t\in [0, T]}|X_{t}^{\vare}-\hat{X}_{t}^{\vare}|^p\Big)\leq C_{p,T}(1+|x|^2)\delta^{\theta[\alpha\wedge (\beta\gamma)]}.
\end{eqnarray*}
\end{lemma}

\begin{proof}
By \eref{FX} and \eref{AMX}, it is to see that
\begin{eqnarray*}
|X^{\vare}_t-\hat{X}^{\vare}_t|\leq\!\!\!\!\!\!\!\!&&\left|\int^t_0 \!\!e^{(t-s)A}\langle D U(X^{\vare}_{s}), B(X^{\vare}_{s}, Y^{\vare}_s)-\bar{B}(X^{\vare}_{s})\rangle\right.\\
&&\quad\quad\left.-e^{(t-s(\delta))A}\langle D U(X^{\vare}_{s(\delta)}), B(X^{\vare}_{s(\delta)}, \hat Y^{\vare}_s)-\bar{B}(X^{\vare}_{s(\delta)})\rangle ds\right|\nonumber\\
&&\!\!+\left|\int^t_0 e^{(t-s)A}\left[B(X^{\vare}_{s},  Y^{\vare}_s)-\bar{B}(X^{\vare}_{s})\right]-e^{(t-s(\delta))A}\left[B(X^{\vare}_{s(\delta)}, \hat Y^{\vare}_s)-\bar{B}(X^{\vare}_{s(\delta)})\right] ds\right|\\
\leq\!\!\!\!\!\!\!\!&&C\int^t_0 \!\!\|e^{(t-s)A}-e^{(t-s(\delta))A}\|ds+C\int^t_0\|D U(X^{\vare}_s)-DU(X^{\vare}_{s(\delta)})\|ds\\
&&+C\int^t_0 \left|B(X^{\vare}_s, Y^{\vare}_s)-\bar{B}(X^{\vare}_s)-B(X^{\vare}_{s(\delta)},\hat Y^{\vare}_s)+\bar{B}(X^{\vare}_{s(\delta)})\right|ds.
\end{eqnarray*}
Then by the boundedness of $D^2 U$ and $B$, the H\"{o}lder continuous of $B$ and $\bar{B}$, property \eref{SP4}, we get
\begin{eqnarray*}
\EE\left(\sup_{t\in[0, T]}|X^{\vare}_t-\hat{X}^{\vare}_t|^p\right)
\leq\!\!\!\!\!\!\!\!&&C_p\delta^{\theta p}\left(\int^T_0 s^{-\theta}ds\right)^p+C_{p,T}\EE\int^T_0|X^{\vare}_s-X^{\vare}_{s(\delta)}|^2 ds\\
&&+C_{p,T}\EE\int^T_0|X^{\vare}_s-X^{\vare}_{s(\delta)}|^{2\alpha}ds+C_{p,T}\EE\int^T_0|X^{\vare}_s-X^{\vare}_{s(\delta)}|^{2(\alpha\wedge(\beta\gamma))}ds\\
&&+C_{p,T}\EE\int^T_0|Y^{\vare}_s-\hat Y^{\vare}_s|^{2\beta}ds
\leq C_{p,T}(1+|x|^2)\delta^{\theta[\alpha\wedge (\beta\gamma)]},
\end{eqnarray*}
where the last inequality follows by Lemmas \ref{COX} and \ref{DEY}.
The proof is complete.
\end{proof}

Now, we give a position to prove our main result.

\vspace{1mm}
\noindent\textbf{Proof of Theorem \ref{main result 1}}: We will divide the proof into three steps.

\noindent\textbf{Step 1}: By \eref{FbarX} and \eref{AMX}, it is easy to see that
\begin{eqnarray*}
\hat{X}^{\vare}_{t}-\bar{X}_{t}=\!\!\!\!\!\!\!\!&&\int^t_0 e^{(t-s)A}\lambda (U(X^{\vare}_s)-U(\bar X_s))ds+U(\bar{X}_t)-U(X^{\vare}_t)\\
&&+\int^t_0 Ae^{(t-s)A}\left[U(\bar{X}_{s})-U(X^{\vare}_{s})\right]ds+\int^t_0 e^{(t-s)A}[DU(X^{\vare}_s)-DU(\bar X_s)]\sqrt{Q_1}dW^1_s\\
&&+\int^t_0 e^{(t-s(\delta))A}\langle DU(X^{\vare}_{s(\delta)})+I, B(X^{\vare}_{s(\delta)},  \hat Y^{\vare}_s)-\bar{B}(X^{\vare}_{s(\delta)})\rangle ds.
\end{eqnarray*}
Then we have the following estimate:
\begin{align}
&\EE\left(\sup_{t\in[0, T]}|\hat{X}^{\vare}_{t}-\bar{X}_{t}|^p\right)\nonumber\\
&\leq C_p\lambda^p T\EE\int^T_0 |U(X^{\vare}_s)-U(\bar X_s)|^pds+C_p\EE\left(\sup_{t\in[0, T]}|U(X^{\vare}_t)-U(\bar X_t)|^p\right)\nonumber\\
&+C_p\EE\left(\sup_{t\in[0, T]}\left|\int^t_0 Ae^{(t-s)A}\Big(U(X^{\vare}_s)-U(\bar X_s)\Big)ds\right|^p\right)\nonumber\\
&+C_p\EE\left(\sup_{t\in[0, T]}\left| \int^t_0 e^{(t-s)A}[DU(X^{\vare}_s)-DU(\bar X_s)]\sqrt{Q_1}dW^1_s\right|^p\right)\nonumber\\
&+C_p\EE\left(\sup_{t\in[0, T]}\left|\int^t_0 e^{(t-s(\delta))A}\langle DU(X^{\vare}_{s(\delta)})+I, B(X^{\vare}_{s(\delta)},  \hat Y^{\vare}_s)-\bar{B}(X^{\vare}_{s(\delta)})\rangle ds\right|^p\right)\nonumber\\
&:=\sum^5_{i=1}J_i(T).\label{J_1-J_5}
\end{align}

For the term $J_1(T)$,  by the H\"{o}lder inequality and \eref{BU} we have
\begin{align}
J_1(T)&\leq C_{p,T}\lambda^p D^p_{\lambda}\|\bar B\|^p_{C^{\alpha\wedge(\beta \gamma)}_b}\int^T_0\EE|X^{\vare}_s-\bar X_s|^p ds.\label{J_1}
\end{align}

For the term $J_2(T)$, using \eref{BU} again, it is easy to see  that
\begin{align}
J_2(T)&\leq C_p D^p_{\lambda}\|\bar B\|^p_{C^{\alpha\wedge(\beta \gamma)}_b}\EE\sup_{t\in [0,T]}|X^{\vare}_t-\bar X_t|^p.\label{J_2}
\end{align}

For the term $J_3(T)$, using the factorization method, for any $\kappa_3 \in (0, \kappa_2)$ in \ref{A5}, we write
$$
\int^t_0 Ae^{(t-s)A}(U(X^{\vare}_s)-U(\bar X_s))ds=\frac{\sin( \pi\kappa_3)}{\pi}\int^t_0 e^{(t-s)A}(t-s)^{\kappa_3-1}f_sds,
$$
where
$$
f_s:=\int^s_0 Ae^{(s-r)A}(s-r)^{-\kappa_3}(U(X^{\vare}_r)-U(\bar X_r))dr.
$$
Choosing $p>1$ large enough such that $\frac{p(1-\kappa_3)}{p-1}<1$, we get
\begin{eqnarray*}
\left|\int^t_0 Ae^{(t-s)A}(U(X^{\vare}_s)-U(\bar X_s))ds\right|\leq\!\!\!\!\!\!\!\!&&C\left(\int^t_0 (t-s)^{-\frac{p(1-\kappa_3)}{p-1}}ds\right)^{\frac{p-1}{p}}|f|_{L^{p}(0,T; H)}\\
\leq\!\!\!\!\!\!\!\!&&C_{p}t^{\kappa_3-\frac{1}{p}}|f|_{L^{p}(0,T; H)},
\end{eqnarray*}
which implies
\begin{eqnarray*}
\sup_{0 \leq t\leq T}\left|\int^t_0 Ae^{(t-s)A}(U(X^{\vare}_s)-U(\bar X_s))ds\right|^{p}\leq\!\!\!\!\!\!\!\!&&C_{p,T}|f|^{p}_{L^{p}(0,T; H)}.
\end{eqnarray*}
Then we can deduce by \eref{BU2} that for $p$ large enough such that $\frac{(1-\kappa_2+\kappa_3)p}{p-1}<1$,
\begin{eqnarray*}
J_3(T)\leq\!\!\!\!\!\!\!\!&&C_{p,T}\EE\int^T_0 |f_t|^p dt\nonumber\\
\leq\!\!\!\!\!\!\!\!&&C_{p,T}\EE\int^T_0\!\!\!\left(\int^t_0 \|(-A)^{1-\kappa_2}e^{(t-r)A}(t-r)^{-\kappa_3}\|^{\frac{p}{p-1}}dr\right)^{p-1}\!\!\!\!\int^t_0 \|(-A)^{\kappa_2}DU\|^p_{\infty} |X^{\vare}_r-\bar X_r|^p dr dt\nonumber\\
\leq\!\!\!\!\!\!\!\!&&C_{p,T}\left(\int^T_0 t^{-\frac{(1-\kappa_2+\kappa_3)p}{p-1}}dt\right)^{p-1}\EE\int^T_0|X^{\vare}_t-\bar X_t|^p dt\nonumber\\
\leq\!\!\!\!\!\!\!\!&&C_{p,T}\EE\int^T_0|X^{\vare}_t-\bar X_t|^p dt.
\end{eqnarray*}

For the term $J_4(T)$, similar as we did in $J_3(T)$, we can prove for $p$ large enough,
\begin{eqnarray*}
J_4(T)\leq\!\!\!\!\!\!\!\!&&C_{p,T}\EE\int^T_0 \left|\int^t_0 e^{(t-s)A}(t-s)^{-\kappa_4}\left[DU(X^{\vare}_s)-DU(\bar X_s)\right]\sqrt{Q_1}dW^1_s\right|^p dt\\
\leq\!\!\!\!\!\!\!\!&&C_{p,T}\int^T_0 \EE\left(\int^t_0 \left\|e^{(t-s)A}(t-s)^{-\kappa_4}\left[DU(X^{\vare}_s)-DU(\bar X_s)\right]\sqrt{Q_1}\right\|^2_{HS}ds \right)^{p/2}dt,
\end{eqnarray*}
where $\kappa_4\in(0,\zeta/2)$, and $\zeta$ is given in \ref{A3}. Note that
\begin{align*}
&\left\|e^{(t-s)A}(t-s)^{-\kappa_4}\Big(DU(X^{\vare}_s)-DU(\bar X_s)\Big)\sqrt{Q_1}\right\|^2_{HS}\nonumber\\
\leq &\sum_{i,j}\langle e^{(t-s)A}(t-s)^{-\kappa_4}\Big(DU(X^{\vare}_s)-DU(\bar X_s)\Big)\sqrt{Q_1} e_i, e_j\rangle^2\nonumber\\
\leq &\sum_{i,j}e^{-2\lambda_j(t-s)}(t-s)^{-2\kappa_4}\left\langle \Big(DU(X^{\vare}_s)-DU(\bar X_s)\Big)\sqrt{Q_1} e_i, e_j\right\rangle^2\nonumber\\
\leq &\sum_{i,j}e^{-2\lambda_j(t-s)}(t-s)^{-2\kappa_4}\left\langle \Big(DU_j(X^{\vare}_s)-DU_j(\bar X_s)\Big), \sqrt{Q_1} e_i\right\rangle^2\nonumber\\
\leq &\sum_{j}e^{-2\lambda_j(t-s)}(t-s)^{-2\kappa_4}\sum_{i}\left\langle \sqrt{Q_1}\Big (DU_j(X^{\vare}_s)-DU_j(\bar X_s)\Big), e_i\right\rangle^2\nonumber\\
\leq &\|Q_1\|\sum_{j}e^{-2\lambda_j(t-s)}(t-s)^{-2\kappa_4}\|D^{2}U_j\|^2_{\infty}|X^{\vare}_s-\bar X_s|^2\nonumber\\
\leq &C\|Q_1\|\|\bar{B}\|^2_{C^{\alpha\wedge(\beta \gamma)}_b}|X^{\vare}_s-\bar X_s|^2\sum_{j}e^{-2\lambda_j(t-s)}(t-s)^{-2\kappa_4},
\end{align*}
where $\langle DU_j(x), y\rangle:=\langle DU(x)y, e_j\rangle$. This and assumption \ref{A3}, we get
\begin{align}
J_4(T)&\leq C_{p,T}\int^T_0 \EE\left(\int^t_0 |X^{\vare}_s-\bar X_s|^2\sum_{j}e^{-2\lambda_j(t-s)}(t-s)^{-2\kappa_4}ds \right)^{p/2}dt\nonumber\\
&\leq  C_{p,T}\int^T_0 \EE\left(\sup_{s\in[0, t]}|X^{\vare}_s-\bar X_s|^p\right)\left(\sum_{j}\int^t_0e^{-2\lambda_j s}s^{-2\kappa_4}ds \right)^{p/2}dt\nonumber\\\
&\leq C_{p,T}\int^T_0 \EE\left(\sup_{s\in[0, t]}|X^{\vare}_s-\bar X_s|^p\right)dt.\label{J_4}
\end{align}

Combining \eref{J_1-J_5}-\eref{J_4}, we get
\begin{align*}
&\EE\left(\sup_{t\in[0,T]}|X^{\vare}_t-\bar X_t|^p\right)\\
\leq & C_p\EE\left(\sup_{t\in[0,T]}|X^{\vare}_t-\hat X^{\vare}_t|^p\right)+C_p\EE\left(\sup_{t\in[0,T]}|\hat X^{\vare}_t-\bar X_t|^p\right)\\
\leq & C_p D^p_{\lambda}\|\bar B\|^p_{C^{\alpha\wedge(\beta \gamma)}_b}\EE\left(\sup_{t\in [0,T]}|X^{\vare}_t-\bar X_t|^p\right)+C_{p,T,\lambda}\int^T_0\EE\left(\sup_{s\in[0,t]}|X^{\vare}_s-\bar X_s|^p\right)dt\\
&+C_p\EE\left(\sup_{t\in[0,T]}|X^{\vare}_t-\hat X^{\vare}_t|^p\right)+\EE J_5(T).
\end{align*}
For any fixed $p,T>0$, since $\lim_{\lambda\rightarrow\infty}D_\lambda=0$, taking $\lambda$ sufficient large such that $C_p D^p_{\lambda}\|\bar B\|^p_{C^{\alpha\wedge(\beta \gamma)}_b}\leq 1/2$,
then we have
\begin{eqnarray*}
\EE\left(\sup_{t\in[0,T]}|X^{\vare}_t-\bar X_t|^p\right)\leq\!\!\!\!\!\!\!\!&& C_{p,T}\int^T_0\EE\left(\sup_{s\in[0,t]}|X^{\vare}_s-\bar X_s|^p\right)dt\\
&&+C_p\EE\left(\sup_{t\in[0,T]}|X^{\vare}_t-\hat X^{\vare}_t|^p\right)+\EE J_5(T).
\end{eqnarray*}
Then the Gronwall inequality yields
\begin{eqnarray}
\EE\left(\sup_{t\in[0,T]}|X^{\vare}_t-\bar X_t|^p\right)\leq C_{p,T}\left[\EE\left(\sup_{t\in[0,T]}|X^{\vare}_t-\hat X^{\vare}_t|^p\right)+\EE J_5(T)\right].\label{L2}
\end{eqnarray}

\noindent\textbf{Step 2:} In this step, we intend to estimate $\EE J_5(T)$.
\begin{eqnarray}
&&\EE J_5(T)\nonumber\\
\leq\!\!\!\!\!\!\!\!&&C_{p,T}\EE\Big[\sup_{t\in[0,T]}\left|\int^t_0 e^{(t-s(\delta))A}\langle DU(X^{\vare}_{s(\delta)})+I, B(X^{\vare}_{s(\delta)},  \hat Y^{\vare}_s)-\bar{B}(X^{\vare}_{s(\delta)})\rangle ds\right|^2\Big]\nonumber\\
\leq\!\!\!\!\!\!\!\!&&C_{p,T}\EE\Big[\sup_{t\in[0,T]}\Big|\sum_{k=0}^{[t/\delta]-1}
\!\!e^{(t-(k+1)\delta)A}\!\!\int_{k\delta}^{(k+1)\delta}\!\!\!\!e^{((k+1)\delta-s(\delta))A}\langle DU(X^{\vare}_{s(\delta)})+I, B(X^{\vare}_{s(\delta)},  \hat Y^{\vare}_s)-\bar{B}(X^{\vare}_{s(\delta)})\rangle ds\nonumber\\
\!\!\!\!\!\!\!\!&&+\int_{t(\delta)}^{t}e^{(t-s(\delta))A}\langle DU(X^{\vare}_{s(\delta)})+I, B(X^{\vare}_{s(\delta)},  \hat Y^{\vare}_s)-\bar{B}(X^{\vare}_{s(\delta)})\rangle ds\Big|^2\Big]\nonumber\\
\leq\!\!\!\!\!\!\!\!&&C_{p,T}\EE\left[\sup_{t\in[0,T]}[t/\delta]\sum_{k=0}^{[t/\delta]-1}
\left|\int_{k\delta}^{(k+1)\delta}\!\!\!e^{((k+1)\delta-s(\delta))A}\langle DU(X^{\vare}_{s(\delta)})+I, B(X^{\vare}_{s(\delta)},  \hat Y^{\vare}_s)-\bar{B}(X^{\vare}_{s(\delta)})\rangle ds\right|^2\right]\nonumber\\
\!\!\!\!\!\!\!\!&&+C_{p,T}\EE\left[\sup_{t\in[0,T]}\left|\int_{t(\delta)}^{t}e^{(t-s(\delta))A}\langle DU(X^{\vare}_{s(\delta)})+I, B(X^{\vare}_{s(\delta)},  \hat Y^{\vare}_s)-\bar{B}(X^{\vare}_{s(\delta)})\rangle ds\right|^2\right]\nonumber\\
:=\!\!\!\!\!\!\!\!&&J_{51}(T)+J_{52}(T).\nonumber
\end{eqnarray}
For the term $J_{52}(T)$, by the boundedness of $B$, $\bar B$ and $DU$, it is easy to see that
\begin{eqnarray}
J_{52}(T)\leq C_{p,T}\delta^{2}.\label{J62}
\end{eqnarray}
For the term $J_{51}(T)$, we have
\begin{eqnarray}
J_{51}(T)\leq\!\!\!\!\!\!\!\!&&C_{p,T}[T/\delta]\sum_{k=0}^{[T/\delta]-1}
\EE\left|\int_{k\delta}^{(k+1)\delta}e^{((k+1)\delta-k\delta)A}\langle DU(X^{\vare}_{k\delta})+I, B(X^{\vare}_{k\delta},  \hat Y^{\vare}_s)-\bar{B}(X^{\vare}_{k\delta})\rangle ds\right|^2\nonumber\\
\leq\!\!\!\!\!\!\!\!&&\frac{C_{p,T}}{\delta^2}\max_{0\leq k\leq[T/\delta]-1}\mathbb{E}\left|\int_{k\delta}^{(k+1)\delta}
B(X^{\vare}_{k\delta},  \hat Y^{\vare}_s)-\bar{B}(X^{\vare}_{k\delta})ds\right|^2\nonumber\\
\leq\!\!\!\!\!\!\!\!&&\frac{C_{p,T}\vare^2}{\delta^2}\max_{0\leq k\leq[T/\delta]-1}\mathbb{E}\left|\int_{0}^{\delta/\vare}
B(X^{\vare}_{k\delta},  \hat Y^{\vare}_{s\vare+k\delta})-\bar{B}(X^{\vare}_{k\delta})ds\right|^2\nonumber\\
=\!\!\!\!\!\!\!\!&&\frac{C_{p,T}\vare^2}{\delta^2}\max_{0\leq k\leq[T/\delta]-1}\int_{0}^{\frac{\delta}{\vare}}
\int_{r}^{\frac{\delta}{\vare}}\Psi_{k}(s,r)dsdr, \label{J61}
\end{eqnarray}
where for any $0\leq r\leq s\leq \frac{\delta}{\vare}$,
\begin{eqnarray*}
\Psi_{k}(s,r):=\!\!\!\!\!\!\!\!&&\mathbb{E}\left[
\langle B(X_{k\delta}^{\vare},\hat{Y}_{s\vare+k\delta}^{\vare})-\bar{B}(X_{k\delta}^{\vare}),
B(X_{k\delta}^{\vare},\hat{Y}_{r\vare+k\delta}^{\vare})-\bar{B}(X_{k\delta}^{\vare})\rangle\right].
\end{eqnarray*}

Then by Proposition \ref{Ergodicity} and following a standard argument (see \cite{LRSX2} for instance), it is easy to see that
\begin{eqnarray}
\Psi_{k}(s,r)\leq\!\!\!\!\!\!\!\!&&C_{T}(|y|^{\beta}+1)e^{-\frac{(s-r)(\lambda_1-L_F)\beta}{2}}.\label{S2}
\end{eqnarray}
As a result, it follows from \eref{J62}-\eref{S2},
\begin{eqnarray}
J_{5}(T)\leq\!\!\!\!\!\!\!\!&& C_{p,T}(1+|y|^{\beta})\frac{\vare^2}{\delta^2}\int_{0}^{\frac{\delta}{\vare}}\int_{r}^{\frac{\delta}{\vare}}e^{-\frac{(s-r)(\lambda_1-L_F)\beta}{2}}dsdr+C_{p,T}\delta^{2}  \nonumber\\
\leq\!\!\!\!\!\!\!\!&&C_{p,T}(1+|y|^{\beta})\left(\frac{\vare^2}{\delta^2}+\frac{\vare}{\delta}+\delta^{2}\right).\label{J_5}
\end{eqnarray}

\noindent\textbf{Step 3:} By the preparation above, we intend to finish the proof in this step, by Lemma \ref{DEX}, \eref{L2} and \eref{J_5}, we have
\begin{eqnarray*}
\mathbb{E}\left(\sup_{t\in [0, T]}|X_{t}^{\vare}-\bar X_{t}|^p\right)\leq\!\!\!\!\!\!\!\!&& C_{p,T}(1+|x|^2+|y|^{\beta})\left(\delta^{\theta[\alpha\wedge (\beta\gamma)]}+\frac{\vare^2}{\delta^2}+\frac{\vare}{\delta}+\delta^{2}\right).
\end{eqnarray*}
Then it follows by taking $\delta=\vare^{\frac{1}{\theta[\alpha\wedge (\beta\gamma)]+1}}$,
\begin{eqnarray*}
\mathbb{E}\left(\sup_{t\in [0, T]}|X_{t}^{\vare}-\bar X_{t}|^p\right)\leq C_{p,T}(1+|x|^2+|y|^{\beta})\vare^{\frac{\theta[\alpha\wedge (\beta\gamma)]}{\theta[\alpha\wedge (\beta\gamma)]+1}},
\end{eqnarray*}
which implies the desired result. The whole proof is complete.

\section{Application to example}\label{Sec Exs}

In this section we will apply our main result to establish the averaging principle for a class of slow-fast SPDEs with H\"{o}lder continuous coefficients.
i.e., considering the following non-linear stochastic heat equation on $D=[0,\pi]$ with Dirichlet boundary conditions:
\begin{equation}\left\{\begin{array}{l}\label{Ex}
\displaystyle
dX^{\varepsilon}(t,\xi)=\left[\Delta X^{\varepsilon}(t,\xi)
+B(X^{\varepsilon}(t,\cdot), Y^{\varepsilon}(t,\cdot))(\xi)\right]dt+(-\Delta)^{-r_1 /2}d W^1(t,\xi),\\
\displaystyle dY^{\varepsilon}(t,\xi)=\frac{1}{\varepsilon}\left[\Delta Y^{\varepsilon}(t,\xi)+F(X^{\varepsilon}(t,\cdot), Y^{\varepsilon}(t,\cdot))(\xi)\right]dt
+\frac{1}{\sqrt{\varepsilon}}(-\Delta)^{-r_2 /2}d W^2(t,\xi),\\
X^{\vare}(t,\xi)=Y^{\vare}(t,\xi)=0, \quad t> 0,\quad \xi\in\partial D,\\
X^{\vare}(0,\xi)=x(\xi),Y^{\vare}(0,\xi)=y(\xi)\quad \xi\in D, x,y\in H,\end{array}\right.
\end{equation}
where $\partial D$ the boundary of $D$, $W^1_t$ and $W^2_t$ are two cylindrical Wiener process in $H:=L^2(D)$ (with Dirichlet boundary conditions). Put
\begin{eqnarray*}
&&Ax=\Delta x,\quad x\in \mathscr{D}(A)=H^{2}(D) \cap H^1_0(D);\\
&&Q_1=(-\Delta)^{-r_1},\quad Q_2=(-\Delta)^{-r_2},\quad r_1,r_2\in (0, 1/7);\\
&&B(x,y)(\xi)=\sin(\sqrt{|x(\xi)|}+\sqrt{|y(\xi)|}),\quad x,y\in H;\\
&&F(x,y)(\xi)=1/2\cos(\sqrt{|x(\xi)|}+|y(\xi)|),\quad x,y\in H.
\end{eqnarray*}
Then it is easy to check that $B$ and $F$ satisfy the assumption \ref{A1}.

The operator $A$ is a self-adjoint operator and possesses a complete orthonormal system of eigenfunctions, namely
$$e_k(\xi)=(\sqrt{2/\pi})\sin(k\xi),\quad\xi\in[0,\pi],$$
where $k\in \mathbb{N}$. The corresponding eigenvalues of $A$ are $-\lambda_k$ with $\lambda_k=k^{2}$. As a result, it is easy to see assumptions \ref{A2} and \ref{A6} are satisfied. Moreover, assumption \ref{A3} holds for any $\zeta\in(0,1/2)$.

For the assumption \ref{A4}, we first note that for $i=1,2$,
$$Q_i(t)=\int^t_0 e^{sA}(-A)^{-r_i}e^{sA^{*}}ds=\frac{1}{2}(-A)^{-(r_i+1)}(I-e^{2tA}).$$
So $Q_i(t)$ is a trace class operator if
$$
\sum^{\infty}_{k=1}\frac{1}{\lambda_k^{r_i+1}}<\infty,
$$
which holds by $r_i>-1/2$.

By a straightforward computer, for any $\theta\in(0, r_1+1/2)$,
$$
\int^T_0 r^{-\theta}\|e^{rA}\sqrt{Q_1}\|^2_{HS}dr\leq \sum^{\infty}_{k=1}\frac{1}{\lambda^{r_1+1-\theta}_k}<\infty,\quad \forall T>0.
$$
$$
\int^T_0 \|(-A)^{\theta/2}e^{rA}\sqrt{Q_1}\|^2_{HS}dr\leq \sum^{\infty}_{k=1}\frac{1}{\lambda^{r_1+1-\theta}_k}<\infty,\quad \forall T>0.
$$
$$
\int^{\infty}_0 \|e^{rA}\sqrt{Q_2}\|^2_{HS}dr\leq \sum^{\infty}_{k=1}\frac{1}{\lambda^{r_2+1}_k}<\infty,
$$
which imply the conditions \eref{A41}-\eref{A43} hold.

To show assumption \ref{A5}, note that
$$
\Lambda_i(t)=Q_i^{-1/2}(t)e^{t A}=\frac{1}{2}(-A)^{(1+r_i)/2}(I-e^{2tA})^{-1/2}e^{tA}.
$$
Then $\Lambda_i(t)$ is a bounded operator for any $t>0$, in fact,
\begin{eqnarray*}
\|\Lambda_i(t)\|=\!\!\!\!\!\!\!\!&&\|\frac{1}{2}(-A)^{(1+r_i)/2}(I-e^{2tA})^{-1/2}e^{tA}\|\\
=\!\!\!\!\!\!\!\!&&\frac{1}{2}t^{-(1+r_i)/2}\|(-tA)^{(1+r_i)/2}(I-e^{2tA})^{-1/2}e^{tA}\|\\
\leq\!\!\!\!\!\!\!\!&&Ct^{-(1+r_i)/2},
\end{eqnarray*}
where we use the fact that the operator $(-tA)^{(1+r_i)/2}(I-e^{2tA})^{-1/2}e^{tA}$ is uniformly bounded because the function $s\rightarrow s^{(1+r_i)/2}(1-e^{-2s})^{-1/2}e^{-s}$ is bounded on $(0,\infty)$.

Furthermore, for any $\lambda>0$
$$
\int^{\infty}_0 e^{-\lambda t}\|\Lambda_i(t)\|^{1+\kappa_1}dt\leq C\int^{\infty}_0 e^{-\lambda t}t^{-(1+r_i)(1+\kappa_1)/2}dt<\infty
$$
holds for $\kappa_1=3/4$ due to $r_i<1/7$. Therefore the condition \eref{A50} holds.

By a similar arguments, we also have for any $\kappa_2\in \left(0, \frac{(1-r_1)}{2}\right)$,
$$
\int^{\infty}_0 e^{-\lambda t}\|(-A)^{\kappa_2}\Lambda_1(t)\|dt\leq C\int^{\infty}_0 e^{-\lambda t}t^{-(1+r_1+2\kappa_2)/2}dt<\infty,
$$
which verifies the assumption \eref{A51} holds. Consequently, by Theorem \ref{main result 1}, the slow component $X^{\vare}$ of the stochastic system \eref{Ex} strongly convergence to the solution $\bar{X}$ of the corresponding averaged equation.

\vspace{0.3cm}
\textbf{Acknowledgment}. This work is supported by the NNSF of China (11601196, 11701233, 11771187, 11931004), NSF of Jiangsu (BK20170226) and the Project Funded by the Priority Academic Program Development of Jiangsu Higher Education Institutions.

\vspace{0.3cm}

\end{document}